\documentclass[11pt]{article}

\usepackage{amsmath,amsthm,amssymb}

\usepackage[pdfauthor={Vaidotas Characiejus, Alfredas Rackauskas},pdftitle={Operator self-similar processes and functional central limit theorems},pdftex,hidelinks]{hyperref}

\usepackage{enumerate}

\newcommand{\fl}[1]{\lfloor#1\rfloor}

\DeclareMathOperator*{\esssup}{ess\,sup}
\DeclareMathOperator*{\essinf}{ess\,inf}

\theoremstyle{plain}
\newtheorem*{thm*}{Theorem}
\newtheorem{thm}{Theorem}
\newtheorem*{prp*}{Proposition}
\newtheorem{prp}{Proposition}
\newtheorem*{lemma*}{Lemma}
\newtheorem{lemma}{Lemma}
\newtheorem*{cor*}{Corollary}

\theoremstyle{definition}
\newtheorem*{dfn*}{Definition}

\theoremstyle{remark}
\newtheorem*{remark*}{Remark}
\newtheorem{remark}{Remark}

\title{Operator self-similar processes and functional central limit theorems}
\author{Vaidotas Characiejus\thanks{Corresponding author. Tel.: +37064034677.}\qquad Alfredas Ra\v ckauskas\\
\small Faculty of Mathematics and Informatics, Vilnius University, Naugarduko 24, 03225 Vilnius, Lithuania\\
\small(e-mail: \href{mailto:vaidotas.characiejus@gmail.com}{vaidotas.characiejus@gmail.com}; \href{mailto:alfredas.rackauskas@mif.vu.lt}{alfredas.rackauskas@mif.vu.lt})}
\date{March 6, 2014}

\begin{document}

\maketitle
\begin{abstract}
Let $\{X_k:k\ge1\}$ be a linear process with values in the separable Hilbert space $L_2(\mu)$ given by $X_k=\sum_{j=0}^\infty(j+1)^{-D}\varepsilon_{k-j}$ for each $k\ge1$, where $D$ is defined by $Df=\{d(s)f(s):s\in\mathbb S\}$ for each $f\in L_2(\mu)$ with $d:\mathbb S\to\mathbb R$ and $\{\varepsilon_k:k\in\mathbb Z\}$ are independent and identically distributed $L_2(\mu)$-valued random elements  with $\operatorname E\varepsilon_0=0$ and $\operatorname E\|\varepsilon_0\|^2<\infty$. We establish sufficient conditions for the functional central limit theorem for $\{X_k:k\ge1\}$ when the series of operator norms $\sum_{j=0}^\infty\|(j+1)^{-D}\|$ diverges and  show that the limit process generates an operator self-similar process.


\end{abstract}

\section{Introduction}
Self-similar processes are stochastic processes that are invariant in distribution under
suitable scaling of time and space. More precisely, let $\xi=\{\xi(t):t\ge0\}$ be an $\mathbb R^q$-valued stochastic process defined on some probability space $(\Omega,\mathcal F,P)$. The process $\xi$ is said to be self-similar if for any $a>0$ there exists $b>0$ such that
$$
\{\xi(at):t\ge0\}\overset{fdd}=\{b\xi(t):t\ge0\},
$$
where $\overset{fdd}=$ denotes the equality of the finite-dimensional distributions.

Self-similar processes were first studied rigorously by Lamperti~\cite{lamperti1962}. Well-known examples are the Brownian motion and the fractional Brownian motion with Hurst parameter $0<H<1$ (in these cases $b$ is equal to $a^{1/2}$ and $a^H$ respectively). We refer to Embrechts and Maejima~\cite{embrechts2002} for the current state of knowledge about self-similar processes and their applications.

Laha and Rohatgi~\cite{laha1981} introduced \emph{operator} self-similar processes taking values in $\mathbb R^q$. They extended the notion of self-similarity to allow scaling by a class of matrices. Such processes were later studied by Hudson and Mason~\cite{hudson1982}, Maejima and Mason~\cite{maejima1994}, Lavancier, Philippe, and Surgailis~\cite{lavancier2009}, Didier and Pipiras~\cite{didier2011} among others.

Matache and Matache~\cite{matache2006} consider and study operator self-similar processes valued in (possibly infinite-dimensional) Banach spaces. Let $\mathbb E$ denote a Banach space and let $L(\mathbb E)$ be the algebra of all bounded linear operators on $\mathbb E$. Matache and Matache~\cite{matache2006} give the following definition.
\begin{dfn*}
An operator self-similar process is a stochastic process $\xi=\{\xi(t):t\ge0\}$ on $\mathbb E$ such that there is a family $\{T(a):a>0\}$ in $L(\mathbb E)$ with the property that for each $a>0$,
$$
\{\xi(at):t\ge0\}\overset{fdd}=\{T(a)\xi(t):t\ge0\}.
$$
\end{dfn*}
The family $\{T(a):a>0\}$ is called the scaling family of operators. If operators $\{T(a):a>0\}$ have the particular form $T(a)=a^GI$, where $G$ is some fixed scalar and $I$ is an identity operator, then a stochastic process is called self-similar instead of operator self-similar.

In this paper, we obtain an example of an operator self-similar process with values in the real separable Hilbert space $L_2(\mu)=L_2(\mathbb S,\mathcal S,\mu)$ of equivalence classes of $\mu$-almost everywhere equal square-integrable functions, where $(\mathbb S,\mathcal S,\mu)$ is a $\sigma$-finite measure space. Our example arises from the functional central limit theorem for a sequence of $L_2(\mu)$-valued random elements.

Let $\{X_k\}=\{X_k:k\ge1\}$ be random elements with values in the separable Banach space~$\mathbb E$ given by
\begin{equation}\label{eq:X_k}
X_k=\sum_{j=0}^\infty u_j\varepsilon_{k-j}
\end{equation}
for each $k\ge1$, where $\{u_j\}=\{u_j:j\ge0\}\subset L(\mathbb E)$ and $\{\varepsilon_k\}=\{\varepsilon_k:k\in\mathbb Z\}$ are independent and identically distributed $\mathbb E$-valued random elements with $\operatorname E\varepsilon_0=0$, $\operatorname E\|\varepsilon_0\|^2<\infty$, where $\|\cdot\|$ is the norm of the Banach space~$\mathbb E$. Let $\{\zeta_n\}=\{\zeta_n(t):t\in[0,1]\}_{n\ge1}$ be random polygonal functions (piecewise linear functions) constructed from the partial sums  $\{S_n\}=\{S_n=X_1+\ldots+X_n:n\ge 1\}$. The asymptotic behaviour of $\{S_n\}$ and that of $\{\zeta_n\}$ strongly depend on the convergence of the series $\sum_{j=0}^\infty\|u_j\|$, where $\|\cdot\|$ is the operator norm. Roughly speaking, if the series $\sum_{j=0}^\infty\|u_j\|$ converges, the asymptotic behaviour of $\{S_n\}$ and $\{\zeta_n\}$ is inherited from $\{\varepsilon_k\}$ (see Merlev\`ede, Peligrad, and Utev~\cite{merlevede1997}, Ra\v ckauskas and Suquet~\cite{rackauskas2010} for more details). However, this is not the case when  $\sum_{j=0}^\infty\|u_j\|=\infty$ (see Ra\v ckauskas and Suquet~\cite{rackauskas2011} and Characiejus and Ra\v ckauskas~\cite{characiejus2013}).

Ra\v ckauskas and Suquet~\cite{rackauskas2011} consider $\{X_k\}$ with values in an abstract separable Hilbert space $\mathbb H$ when $\sum_{j=0}^\infty\|u_j\|=\infty$ with $u_0=I$ and $u_j=j^{-T}$ for $j\ge1$, where $T\in L(\mathbb H)$ satisfies $\frac12I<T<I$ and $T$ commutes with the covariance operator of $\varepsilon_0$. We obtain an operator self-similar process with the covariance structure different from Ra\v ckauskas and Suquet~\cite{rackauskas2011} since $T$ does not necessarily commute with the covariance operator of $\varepsilon_0$ in our case.

Specifically, we investigate $\{X_k\}$ with values in $L_2(\mu)$ and $\{u_j\}$ given by 
\begin{equation}\label{eq:u_j}
u_j=(j+1)^{-D}
\end{equation}
for each $j\ge0$, where $D$ is a multiplication operator defined by $Df=\{d(s)f(s):s\in\mathbb S\}$ for each $f\in L_2(\mu)$ with a measurable function $d:\mathbb S\to\mathbb R$. Our main results (Theorem~\ref{thm:fclt1} and Theorem~\ref{thm:fclt2} in Section~\ref{sec:fclt}) establish  sufficient conditions for the convergence in distribution of $\zeta_n$ in the space $C([0,1];L_2(\mu))$ in the following two cases: either $d\in(1/2,1)$ (shorthand for $1/2<d(s)<1$ for all $s\in\mathbb S$) or $d=1$ (shorthand for $d(s)=1$ for all $s\in\mathbb S$). In the former case, we provide sufficient conditions for the convergence in distribution of $n^{-H}\zeta_n$ to a Gaussian stochastic process $\mathcal G$, where $\{n^{-H}\}$ are multiplication operators given by $n^{-H}f=\{n^{-[3/2-d(s)]}f(s): s\in \mathbb S\}$ for each $n\ge 1$ and $f\in L_2(\mu)$. In the latter case, we establish convergence in distribution of $(\sqrt n\log n)^{-1}\zeta_n$ to an $L_2(\mu)$-valued Wiener process.  The results of this paper generalize our previous results since in Characiejus and Ra\v ckauskas~\cite{characiejus2013} only the central limit theorem is investigated.

The rest of the paper is organized as follows. In Section~\ref{sec:pre},  we give two alternative ways to construct $\{X_k\}$ and establish some properties of $\{X_k\}$ and $\{\zeta_n\}$. The existence of an operator self-similar process $\mathcal X$ with values in $L_2(\mu)$ is established in Section~\ref{sec:oss}. In Section~\ref{sec:fclt}, we establish sufficient conditions for the functional central limit theorem.

\section{Preliminaries}\label{sec:pre}
\subsection{Construction of $\{X_k\}$}
There are two approaches to construct $\{X_k\}$ with values in $L_2(\mu)$. The first approach is to define~$\{X_k\}$ as stochastic processes with space varying memory and square $\mu$-integrable sample paths. The second approach is to define $L_2(\mu)$ valued random variable $X_k$ for each $k\ge1$ as series~\eqref{eq:X_k} with $u_j$ given by~\eqref{eq:u_j} and to investigate the convergence of such series. We present both of these two approaches.

\subsubsection*{First approach}
Let $\{\varepsilon_k\}=\{\varepsilon_k(s):s\in\mathbb S\}_{k\in \mathbb Z}$ be independent and identically distributed measurable stochastic processes defined on the probability space $(\Omega,\mathcal F,P)$, i.e.\ $\{\varepsilon_k\}$ are $\mathcal F\otimes\mathcal S$-measurable functions $\varepsilon_k:\Omega\times \mathbb S\to\mathbb R$. We require that $\operatorname E\varepsilon_0(s)=0$ and $\operatorname E\varepsilon_0^2(s)<\infty$ for each $s\in\mathbb S$ and denote
$$
\sigma(r,s)=\operatorname E[\varepsilon_0(r)\varepsilon_0(s)],\quad\sigma^2(s)=\operatorname E\varepsilon_0^2(s),\quad r,\,s\in\mathbb S.
$$

Define stochastic processes $\{X_k\}=\{X_k(s):s\in\mathbb S\}_{k\ge1}$ by setting
\begin{equation}\label{eq:series}
X_k(s)=\sum_{j=0}^\infty(j+1)^{-d(s)}\varepsilon_{k-j}(s)
\end{equation}
for each $s\in\mathbb S$ and each $k\ge1$. Observe that $d(s)>1/2$ is a necessary and sufficient condition for the almost sure convergence of series \eqref{eq:series} (this fact follows from Kolmogorov's three-series theorem). It is well-known that the growth rate of the partial sums $\{\sum_{k=1}^nX_k(s)\}$ depends on $d(s)$. Viewing $\mathbb S$ as the set of space indexes and $\mathbb Z$ as the set of time indexes, we thus have a functional process $\{X_k\}$ with space varying memory. We refer to Giraitis, Koul, and Surgailis~\cite{giraitis2012} for an encyclopedic treatment of long memory phenomenon of stochastic processes.

We denote 
$$
\gamma_h(r,s)=\operatorname E[X_0(r)X_h(s)],\quad\gamma_h(s)=\operatorname E[X_0(s)X_h(s)],\quad r\,,s\in\mathbb S,\quad h\in\mathbb N.
$$
For fixed $r, s\in \mathbb S$, the sequences $\{X_k(r)\}$ and $\{X_k(s)\}$ are stationary sequences of random variables with zero means and cross-covariance
\begin{equation}\label{eq:cross-co}
\gamma_h(r,s)=\sigma(r,s)\sum_{j=0}^\infty(j+1)^{-d(r)}(j+h+1)^{-d(s)}.
\end{equation}

Throughout the paper
\begin{equation}\label{eq:d(r,s)}
d(r,s)=d(r)+d(s),\quad r,\, s\in\mathbb S,
\end{equation}
and
\begin{equation}\label{eq:const}
c(r,s)=\int_0^\infty x^{-d(r)}(x+1)^{-d(s)}\mathrm dx,\quad r,\,s\in\mathbb S,
\end{equation}
provided that $1/2<d(r)<1$, $d(s)>1/2$. Let us observe that $c(r,s)=\mathrm B(1-d(r),d(r,s)-1)$, where $\mathrm B$ is the beta function. If $r=s$, we denote $c(r,s)$ by $c(s)$. $c(s)$ can be estimated from above with the following inequality
\begin{equation}\label{eq:c(s)}
c(s)\le\frac1{1-d(s)}+\frac1{2d(s)-1}.
\end{equation}

Proposition~\ref{prp:assbeh} gives the asymptotic behaviour of  $\gamma_h(r,s)$ and Proposition~\ref{prp:lm} provides a necessary and sufficient condition for the summability of the series $\sum_{k=0}^\infty\gamma_k(r,s)$ (for the proof, see Characiejus and Ra\v ckauskas~\cite{characiejus2013}). The notation $a_n\sim b_n$ indicates that the ratio of the two sequences tends to $1$ as $n\to\infty$. 
\begin{prp}\label{prp:assbeh}
If $1/2<d(r)<1$ and $d(s)>1/2$, then
$$
\gamma_h(r,s)\sim c(r,s)\sigma(r,s)\cdot h^{1-d(r,s)}.
$$
If $d(r)=d(s)=1$, then
$$
\gamma_h(r,s)\sim\sigma(r,s)\cdot h^{-1}\log h.
$$
\end{prp}
\begin{prp}\label{prp:lm}
The series
$$
\sum_{k=0}^\infty\gamma_k(r,s)
$$
converges if and only if $d(r)>1$ and $d(r,s)>2$.
\end{prp}
\begin{remark}
The series $\sum_{k=0}^\infty\gamma_k(s)$ converges if and only if $d(s)>1$.
\end{remark}

Let $\mathcal L_2(\mu)=\mathcal L_2(\mathbb S,\mathcal S,\mu)$ be a separable space of real valued square $\mu$-integrable functions with a seminorm
$$\|f\|=\Bigl[\int_\mathbb S|f(v)|^2\mu(\mathrm dv)\Bigr]^{1/2},\quad f\in\mathcal L_2(\mu),$$
and let $L_2(\mu)=L_2(\mathbb S,\mathcal S,\mu)$ be the corresponding Hilbert space of equivalence classes of $\mu$-almost everywhere equal functions with an inner product
$$\langle f,g\rangle=\int_\mathbb Sf(v)g(v)\mu(\mathrm dv),\quad f,g\in L_2(\mu).$$
With an abuse of notation, we denote by $f$ both a function and its equivalence class to avoid cumbersome notation. The intended meaning should be clear from the context.

Proposition~\ref{prp:samplepaths} establishes a necessary and sufficient condition for the sample paths of the stochastic process $\{X_k(s):s\in\mathbb S\}$ to be almost surely square $\mu$-integrable with $\operatorname E\|X_k\|^2<\infty$ for each $k\ge 1$ (see Characiejus and Ra\v ckauskas~\cite{characiejus2013} for the proof).
\begin{prp}\label{prp:samplepaths}
The sample paths of the stochastic process $\{X_k(s):s\in\mathbb S\}$ almost surely belong to the space $\mathcal L_2(\mu)$ and $\operatorname E\|X_k\|^2<\infty$ for each $k\in\mathbb Z$ if and only if both of the integrals
$$
\operatorname E\|\varepsilon_0\|^2=\int_\mathbb S\sigma^2(v)\mu(\mathrm dv)\quad\text{and}\quad\int_\mathbb S\frac{\sigma^2(v)}{2d(v)-1}\mu(\mathrm dv)
$$
are finite.
\end{prp}

A stochastic process $\{\xi(s):s\in\mathbb S\}$ defined on a probability space $(\Omega,\mathcal F,P)$ with sample paths in $\mathcal L_2(\mu)$ induces the $\mathcal F-\mathcal B(L_2(\mu))$-measurable function $\omega\to\{\xi(s)(\omega):s\in \mathbb S\}:\Omega\to L_2(\mu)$, where $\mathcal B(L_2(\mu))$ is the Borel $\sigma$-algebra of $L_2(\mu)$ (for more details, see Cremers and Kadelka~\cite{cremers1984}). Therefore  we shall  frequently  consider  each stochastic process $\{\xi(s):s\in \mathbb S\}$ with sample paths in  $\mathcal L_2(\mu)$ as a random element with values in $L_2(\mu)$ and denote it by $\{\xi(s): s\in \mathbb S\}$  or simply by $\xi$.

\subsubsection*{Second approach}
Now we establish a necessary and sufficient condition for the mean square convergence of series~\eqref{eq:X_k} with $\{u_j\}$ given by~\eqref{eq:u_j}. Recall that $(j+1)^{-D}f=\{(j+1)^{-d(s)}f(s):s\in\mathbb S\}$ for each $j\ge0$ and $f\in L_2(\mu)$ since $e^T=\sum_{j=0}^\infty T^j/j!$ and $\lambda^T=e^{T\log\lambda}$ for $T\in L(\mathbb E)$ and $\lambda>0$.

\begin{prp}\label{prp:X_k}
Series~\eqref{eq:X_k} with $u_j$ given by~\eqref{eq:u_j} and $L_2(\mu)$-valued random elements $\{\varepsilon_k\}$ such that $\operatorname E\varepsilon_0=0$ and $\operatorname E\|\varepsilon_0\|^2<\infty$ converges in mean square if and only if there exists a measurable set $\mathbb S_0\subset\mathbb S$ such that $\mu(\mathbb S\setminus\mathbb S_0)=0$, $d(s)>1/2$ for all $s\in\mathbb S_0$ and the integral
$$
\int_\mathbb S\frac{\sigma^2(v)}{2d(v)-1}\mu(\mathrm dv)
$$
is finite.
\end{prp}
\begin{proof}
Let $N>M$, $\sigma^2(s)=\operatorname E\varepsilon_0^2(s)$, $s\in\mathbb S$, and observe that
$$
\operatorname E\biggl\|\sum_{j=M+1}^Nu_{j}\varepsilon_{j-k}\biggr\|^2=\sum_{j=M+1}^N\int_\mathbb S(j+1)^{-2d(v)}\sigma^2(v)\mu(\mathrm dv).
$$
Since 
$$
\sum_{j=0}^\infty\int_\mathbb S(j+1)^{-2d(r)}\sigma^2(r)\mu(\mathrm dr)
	=\int_\mathbb S\sum_{j=1}^\infty j^{-2d(r)}\sigma^2(r)\mu(\mathrm dr)
$$
and
$$
\frac1{2d(r)-1}\le\sum_{j=1}^\infty j^{-2d(r)}\le1+\frac1{2d(r)-1}
$$
we have that
$$
\int_\mathbb S\frac{\sigma^2(r)}{2d(r)-1}\mu(\mathrm dr)
	\le\int_\mathbb S\sigma^2(r)\sum_{j=1}^\infty j^{-2d(r)}\mu(\mathrm dr)
	\le\operatorname E\|\varepsilon_0\|^2+\int_\mathbb S\frac{\sigma^2(r)}{2d(r)-1}\mu(\mathrm dr)
$$
and the proof is complete.
\end{proof}
\begin{remark}
Since $\{\varepsilon_k\}$ are independent, it follows from L\'evy-It\^o-Nisio theorem (see Ledoux and Talagrand~\cite{ledoux1991}, Theorem~6.1, p.~151) and Proposition~\ref{prp:X_k} that series~\eqref{eq:X_k} also converges almost surely. Hence, $X_k$ for each $k\ge1$ is an $L_2(\mu)$-valued random element and Proposition~\ref{prp:X_k} is consistent with Proposition~\ref{prp:samplepaths}.
\end{remark}
\begin{remark}
Since $u_j$ given by~\eqref{eq:u_j} are multiplication operators from $L_2(\mu)$ to $L_2(\mu)$, we have that the operator norm $\|u_j\|=\inf\{c>0:\mu(s\in\mathbb S:|(j+1)^{-d(s)}|>c)=0\}$. If $\underline d=\essinf d=1/2$, then we have that $\sum_{j=0}^\infty\|u_j\|^2=\sum_{j=1}^\infty j^{-1}=\infty$, but series \eqref{eq:X_k} might still converge. The square summability of the operator norms of $u_j$ is not a necessary condition for the almost sure convergence of series \eqref{eq:X_k}.
\end{remark}

\subsection{Random polygonal functions $\{\zeta_n\}$}
Let 
$\{\zeta_n\}=\{\zeta_n(t):t\in[0,1]\}_{n\ge1}$ be random polygonal functions (piecewise linear functions) constructed from partial sums  $S_k=X_1+\cdots+X_k, k\ge 1$:
$$
\zeta_n(t)=S_{\fl{nt}}+\{nt\}X_{\fl{nt}+1}
$$
for each $n\ge1$ and each $t\in[0,1]$, where $\fl{\cdot}$ is the floor function defined by $\fl{x}=\max\{m\in\mathbb Z\mid m\le x\}$ for $x\in\mathbb R$ and $\{x\}=x-\fl{x}$ is the fractional part of $x\in\mathbb R$. We adopt the usual convention that an empty sum equals $0$.
For each $t\in[0, 1]$ the random function  $\zeta_n(t)$ can be expressed as a series
$$
\zeta_n(t)=\sum_{j=-\infty}^{\fl{nt}+1}a_{nj}(t)\varepsilon_j,
$$
where
\begin{equation}\label{eq:opweights}
a_{nj}(t)=\sum_{k=1}^{\fl{nt}}v_{k-j}+\{nt\}v_{\fl{nt}+1-j}
\end{equation}
and
\begin{equation}\label{eq:opweightsv}
v_j=
\begin{cases}
u_j,&\mbox{ if }j\ge0;\\
0,&\mbox{ if }j<0.
\end{cases}
\end{equation}
Denote $\zeta_n(s,t)=\sum_{k=1}^{\lfloor nt\rfloor}X_k(s)+\{nt\}X_{\fl{nt}+1}(s)$ for $s\in\mathbb S$ and $t\in[0,1]$.
Each random variable ${\zeta_n(s,t)}$ can be expressed as a series  $\zeta_n(s,t)=\sum_{j=-\infty}^{\fl{nt}+1}a_{nj}(s,t)\varepsilon_j(s)$,
where
\begin{equation}\label{eq:weights}
a_{nj}(s,t)=\sum_{k=1}^{\fl{nt}}v_{k-j}(s)+\{nt\}v_{\fl{nt}+1-j}(s)
\end{equation}
and
\begin{equation}\label{eq:weightsv}
v_j(s)=
\begin{cases}
(j+1)^{-d(s)},&\mbox{ if }j\ge0;\\
0,&\mbox{ if }j<0.
\end{cases}
\end{equation}
Observe that $v_j=v_j(s)$ if $d=1$ since $u_j=(j+1)^{-1}$ if $d=1$. Notice that the upper bounds of summation of the series in the expressions of $\zeta_n(t)$ and $\zeta_n(s,t)$ can be extended up to $\infty$ since $a_{nj}(s,t)=0$ and $a_{nj}(t)=0$ if $j>\fl{nt}+1$.

Set $\mathbb T=\mathbb S\times [0, \infty)$ and define the function $V:\mathbb T^2\to\mathbb R$  by
\begin{multline}\label{eq:covV}
V((r,t),(s,u))=\\
	=\frac{\sigma(r,s)}{[2-d(r,s)][3-d(r,s)]}\bigl[c(s,r)t^{3-d(r,s)}+c(r,s)u^{3-d(r,s)}-C(r,s;t-u)|t-u|^{3-d(r,s)}\bigr],
\end{multline}
where $d(r,s)$ is given by~\eqref{eq:d(r,s)}, $c(r,s)$ is given by~\eqref{eq:const} and
$$
C(r,s;t)=
\begin{cases}
c(r,s)&\mbox{ if }t<0;\\
c(s,r)&\mbox{ if }t>0.
\end{cases}
$$

Now we are prepared to derive  the asymptotic behavior of the sequence of cross-covariances of~$\zeta_n$.
\begin{prp}\label{prp:cov}
Suppose either $1/2<d(r)<1$ and $1/2<d(s)<1$ or $d(r)=d(s)=1$. In both cases, the following asymptotic relation holds
$$\operatorname E[\zeta_n(r,t)\zeta_n(s,u)]\sim\operatorname E[S_{\fl{nt}}(r)S_{\fl{nu}}(s)].$$
\end{prp}
\begin{prp}\label{prp:cov2}
If $1/2<d(r)<1$ and $1/2<d(s)<1$, then
$$\operatorname E[S_{\fl{nt}}(r)S_{\fl{nu}}(s)]\sim V((r,t),(s,u))\cdot n^{3-d(r,s)}$$
for $(r,t),(s,u)\in\mathbb S\times[0,1]$, where $V$ is given by \eqref{eq:covV}.

If $d(r)=d(s)=1$, then
$$\operatorname E[S_{\fl{nt}}(r)S_{\fl{nu}}(s)]\sim\sigma(r,s)\cdot \min(t,u)\cdot n\log^2n.$$
\end{prp}

\begin{remark}
Let us assume that $r=s$ and $1/2<d(s)<1$. By setting $r=s$ in Proposition \ref{prp:cov2} and using Proposition \ref{prp:cov}, we obtain that
$$
\operatorname E[\zeta_n(s,t)\zeta_n(s,u)]
	\sim\frac{\sigma^2(s)c(s)}{[1-d(s)][3-2d(s)]}\cdot\operatorname E[B_{3/2-d(s)}(t)B_{3/2-d(s)}(u)]\cdot n^{3-2d(s)},
$$
where
$$
\operatorname E[B_{3/2-d(s)}(t)B_{3/2-d(s)}(u)]=\frac12[t^{3-2d(s)}+u^{3-2d(s)}-|t-u|^{3-2d(s)}]
$$
is the covariance function of the fractional Brownian motion ${B_{3/2-d(s)}}=\{B_{3/2-d(s)}(t):t\in[0,1]\}$ with the Hurst parameter $3/2-d(s)$ and $c(s)=c(s,s)$ is given by \eqref{eq:const}.
\end{remark}
\begin{remark}\label{remark:var}
The asymptotic behaviour of the variance $\operatorname E\zeta_n^2(s,t)$ follows from Proposition \ref{prp:cov} and Proposition \ref{prp:cov2} by setting $r=s$ and $t=u$: if $1/2<d(s)<1$, then
$$
\operatorname E\zeta_n^2(s,t)\sim\frac{c(s)\sigma^2(s)}{[1-d(s)][3-2d(s)]}\cdot t^{3-2d(s)}\cdot n^{3-2d(s)};
$$
if $d(s)=1$, then
$$
\operatorname E\zeta_n^2(s,t)\sim\sigma^2(s)\cdot t\cdot n\log^2n.
$$
\end{remark}

\begin{proof}[Proof of Proposition \ref{prp:cov2}]
Suppose $t<u$ and split the cross-covariance of the partial sums into two terms
\begin{equation}\label{eq:cov}
\operatorname E\bigl[S_{\fl{nt}}(r)S_{\fl{nu}}(s)\bigr]=\operatorname E\bigl[S_{\fl{nt}}(r)S_{\fl{nt}}(s)\bigr]
	+\operatorname E\bigl[S_{\fl{nt}}(r)\bigl[S_{\fl{nu}}(s)-S_{\fl{nt}}(s)\bigr]\bigr].
\end{equation}

The following two asymptotic relations are proved in Characiejus and Ra\v ckauskas~\cite{characiejus2013}: if $1/2<d(r)<1$ and $1/2<d(s)<1$, then
\begin{equation}\label{eq:asymptotic1}
\operatorname E\bigl[S_n(r)S_n(s)\bigr]\sim\frac{[c(r,s)+c(s,r)]\sigma(r,s)}{[2-d(r,s)][3-d(r,s)]}\cdot n^{3-d(r,s)};
\end{equation}
if $d(r)=d(s)=1$, then
\begin{equation}\label{eq:asymptotic2}
\operatorname E\bigl[S_n(r)S_n(s)\bigr]\sim\sigma(r,s)\cdot n\log^2n.
\end{equation}
The asymptotic behaviour of the first term of sum \eqref{eq:cov} is established using \eqref{eq:asymptotic1} and \eqref{eq:asymptotic2}: if $1/2<d(r)<1$ and $1/2<d(s)<1$, then
\begin{equation}\label{eq:asrel=1}
\operatorname E[S_{\fl{nt}}(r)S_{\fl{nt}}(s)]\sim\frac{[c(r,s)+c(s,r)]\sigma(r,s)}{[2-d(r,s)][3-d(r,s)]}\cdot t^{3-d(r,s)}\cdot n^{3-d(r,s)};
\end{equation}
if $d(r)=d(s)=1$, then
\begin{equation}\label{eq:asrel=2}
\operatorname E[S_{\fl{nt}}(r)S_{\fl{nt}}(s)]\sim\sigma(r,s)\cdot t\cdot n\log^2n.
\end{equation}

In order to establish the asymptotic behaviour of the second term of sum~\eqref{eq:cov}, we express it in the following way
\begin{multline}\label{eq:secondterm1}
\operatorname E\bigl[S_{\fl{nt}}(r)[S_{\fl{nu}}(s)-S_{\fl{nt}}(s)]\bigr]
	=\sum_{k=1}^{m_n-1}k[\gamma_k(r,s)+\gamma_{\fl{nu}-k}(r,s)]\\
	+m_n\sum_{k=0}^{|\fl{nu}-2\fl{nt}|}\gamma_{m_n+k}(r,s),
\end{multline}
where $m_n\mathrel{\mathop:}=\min(\fl{nt},\fl{nu}-\fl{nt})$ (we also use the notation $m\mathrel{\mathop:}=\min(t,u-t)$). For simplicity, denote 
$$
\kappa(a,b)=\sum_{k=a+1}^{b}\gamma_k(r,s)\quad\text{and}\quad \nu(a,b)=\sum_{k=a+1}^{b}k\gamma_k(r,s).
$$
Then we have that
\begin{equation}\label{eq:secondterm2}
\sum_{k=1}^{m_n-1}k\gamma_{\fl{nu}-k}(r,s)=\fl{nu}\kappa(\lfloor nu\rfloor-m_n,\lfloor nu\rfloor-1)-\nu(\lfloor nu\rfloor-m_n,\lfloor nu\rfloor-1).
\end{equation}

Let us a recall a few facts about sequences. We use these facts to establish asymptotic behaviour of the sums in \eqref{eq:secondterm1} and \eqref{eq:secondterm2}. Suppose $\{a_n\}$ and $\{b_n\}$ are sequences of positive real numbers such that $a_n\sim b_n$. Then $\sum_{k=1}^na_k\sim\sum_{k=1}^nb_k$ provided either of these partial sums diverges. Let $f$ be a continuous strictly increasing or strictly decreasing function such that $f(x)/f(x+1)\to1$ as $x\to\infty$ and $\int_1^nf(x)\mathrm dx\to\infty$ as $n\to\infty$. Then $\sum_{k=1}^nf(k)\sim\int_1^nf(x)\mathrm dx$.

Since $\gamma_k(r,s)\sim c(r,s)\sigma(r,s)\cdot k^{1-d(r,s)}$ if $1/2<d(r)<1$ and $d(s)>1/2$ (see Proposition~\ref{prp:assbeh}), we obtain the following asymptotic relations using the facts about sequences mentioned above:
\begin{align}\label{eq:asrel1}
\nu(0, m_n-1)&\sim\frac{c(r,s)\sigma(r,s)m^{3-d(r,s)}}{3-d(r,s)}\cdot n^{3-d(r,s)};\\
\fl{nu}\kappa(\fl{nu}-m_n, \fl{nu}-1)&\sim\frac{c(r,s)\sigma(r,s)u[u^{2-d(r,s)}-(u-m)^{2-d(r,s)}]}{2-d(r,s)}\cdot n^{3-d(r,s)};\\
\nu(\fl{nu}-m_n, \fl{nu}-1)&\sim\frac{c(r,s)\sigma(r,s)[u^{3-d(r,s)}-(u-m)^{3-d(r,s)}]}{3-d(r,s)}\cdot n^{3-d(r,s)};
\end{align}
\begin{multline}\label{eq:asrel4}
m_n\kappa(m_n-1, m_n+|\fl{nu}-2\fl{nt}|)\sim\\
	\sim\frac{c(r,s)\sigma(r,s)m[(m+|u-2t|)^{2-d(r,s)}-m^{2-d(r,s)}]}{2-d(r,s)}\cdot n^{3-d(r,s)}.
\end{multline}
We have that
\begin{multline}\label{eq:asrelcom}
\operatorname E\bigl[S_{\fl{nt}}(r)[S_{\fl{nu}}(s)-S_{\fl{nt}}(s)]\bigr]\sim\\
	\sim\frac{c(r,s)\sigma(r,s)}{[2-d(r,s)][3-d(r,s)]}\bigl[-t^{3-d(r,s)}+u^{3-d(r,s)}-(u-t)^{3-d(r,s)}\bigr]\cdot n^{3-d(r,s)}
\end{multline}
using asymptotic relations~\eqref{eq:asrel1}-\eqref{eq:asrel4}. Combining \eqref{eq:asrel=1} with \eqref{eq:asrelcom}, we obtain
\begin{multline*}
\operatorname E[S_{\fl{nt}}(r)S_{\fl{nu}}(s)]\sim\\
\sim\frac{\sigma(r,s)}{[2-d(r,s)][3-d(r,s)]}\bigr[c(s,r)t^{3-d(r,s)}+c(r,s)[u^{3-d(r,s)}-(u-t)^{3-d(r,s)}]\bigl]\cdot n^{3-d(r,s)}.
\end{multline*}

Similarly, if $d(r)=d(s)=1$, then $\gamma_k(r,s)\sim \sigma(r,s)\cdot k^{-1}\log k$
(see Proposition~\ref{prp:assbeh})
and the following asymptotic relations are true
\begin{align}
\label{eq:asrel=1,1}
\nu(0, m_n-1)&\sim\sigma(r,s)m\cdot n\log n;\\
\fl{nu}\kappa(\fl{nu}-m_n, \fl{nu}-1)&\sim\sigma(r,s)[\log u-\log(u-m)]u\cdot n\log n;\\
\nu(\fl{nu}-m_n, \fl{nu}-1)&\sim\sigma(r,s)m\cdot n\log n;\\
\label{eq:asrel=1,4}
m_n\kappa(m_n-1, m_n+|\fl{nu}-2\fl{nt}|)&\sim\sigma(r,s)[\log(m+|u-2t|)-\log m]m\cdot n\log n.
\end{align}
Since sequences \eqref{eq:asrel=1,1}-\eqref{eq:asrel=1,4} grow slower than sequence \eqref{eq:asrel=2}, we conclude that
$$\operatorname E\bigl[S_{\fl{nt}}(r)S_{\fl{nu}}(s)\bigr]\sim\sigma(r,s)\cdot t\cdot n\log^2n.$$

If $t>u$, the proof is exactly the same as in the case of $t<u$. If $t=u$, then we just use asymptotic relations \eqref{eq:asrel=1} and \eqref{eq:asrel=2}. The proof of Proposition~\ref{prp:cov2} is complete.
\end{proof}

\begin{proof}[Proof of Proposition \ref{prp:cov}]
We have that
\begin{align*}
\operatorname E[\zeta_n(r,t)\zeta_n(s,u)]&=\operatorname E[S_{\fl{nt}}(r)S_{\fl{nu}}(s)]\\
	&\qquad+\{nu\}\operatorname E[S_{\fl{nt}}(r)X_{\fl{nu}+1}(s)]\\
	&\qquad+\{nt\}\operatorname E[S_{\fl{nu}}(s)X_{\fl{nt}+1}(r)]\\
	&\qquad+\{nt\}\{nu\}\operatorname E[X_{\fl{nt}+1}(r)X_{\fl{nu}+1}(s)]
\end{align*}
and
$$\operatorname E[S_{\fl{nt}}(r)X_{\fl{nu}+1}(s)]\le\fl{nt}\gamma_0(r,s).$$
The result follows from Proposition \ref{prp:cov2} since $\operatorname E[S_{\fl{nt}}(r)S_{\fl{nu}}(s)]$ is the only term in the expression of  $\operatorname E[\zeta_n(r,t)\zeta_n(s,u)]$ that grows faster than linearly.
\end{proof}

\section{Operator self-similar process}\label{sec:oss}
In this section, we show that there exists a Gaussian stochastic process $\mathcal X=\{\mathcal X(s,t):(s,t)\in\mathbb T\}$ with zero mean and covariance function $V$ given by \eqref{eq:covV}.  The stochastic process $\{\mathcal X(\cdot,t):t\in[0,\infty)\}$ is an operator self-similar process with values in $L_2(\mu)$.

We begin by showing that the function $V$ is a covariance function.
\begin{prp}
\label{prp:covV}
The function $V: \mathbb T\times\mathbb T\to\mathbb R$, given by \eqref{eq:covV}, with $d\in (1/2, 1)$
is a covariance function of a stochastic process indexed by the set $\mathbb T$.
\end{prp}
\begin{proof}
It follows from equation \eqref{eq:covV} that the function $V$ is symmetric, i.e.\ 
$$V(\tau,\tau')=V(\tau',\tau),\quad\tau,\tau'\in\mathbb T.$$ 
So we need to prove that  the function $V$ is positive definite. Let $N\in\mathbb N$, $\tau_i=(s_i,t_i)\in\mathbb T$ and $w_i\in\mathbb R$, where $i\in\{1,\ldots,N\}$. Denote $M=\max\{t_1,\ldots,t_N\}$ and $\widetilde w_i=w_iM^{3/2-d(s_i)}$, $i\in\{1,\ldots,N\}$. Using equation~\eqref{eq:covV} and  Propositions~\ref{prp:cov} and~\ref{prp:cov2}, we obtain that
\begin{align*}
\sum_{i=1}^N\sum_{j=1}^Nw_iw_jV(\tau_i,\tau_j)
	&=\sum_{i=1}^N\sum_{j=1}^N w_i w_jM^{3-[d(s_i)+d(s_j)]}V((s_i,t_i/M),(s_j,t_j/M))\\
	&=\sum_{i=1}^N\sum_{j=1}^N\widetilde w_i\widetilde w_j\lim_{n\to\infty}\frac1{n^{3-[d(s_i)+d(s_j)]}}\operatorname E[\zeta_n(s_i,t_i/M)\zeta_n(s_j,t_j/M)]\ge0
\end{align*}
since
$$\frac1{n^{3-d(r,s)}}\operatorname E[\zeta_n(r,t)\zeta_n(s,u)]$$
is a covariance function for all $(r,t),(s,u)\in\mathbb S\times[0,1]$ and for all $n\in\mathbb N$.
\end{proof}

Let us recall that a  random element $\xi$ with values in a separable Banach space $\mathbb E$ is Gaussian if for any continuous linear functional $f$ on $\mathbb E$, $f(\xi)$ is real valued Gaussian random variable. A stochastic process $\{\xi_t:t\in T\}$ with values in $\mathbb E$  is  Gaussian if each finite linear combination $\sum_i\alpha_i\xi_{t_i}$, $\alpha_i\in\mathbb R$, $t_i\in T$, is Gaussian random element in $\mathbb E$ (for more details about Gaussian random elements and Gaussian stochastic processes with values in Banach spaces, see the textbook by Ledoux and Talagrand~\cite{ledoux1991}).

We have the following corollary of Proposition \ref{prp:covV}.
\begin{cor*}
There exists a zero mean Gaussian stochastic process $\mathcal X=\{\mathcal X(s,t):(s, t)\in \mathbb T\}$ with the covariance function $V$ given by \eqref{eq:covV}.
\end{cor*}

Next we describe the sample path properties of the stochastic process $\mathcal X.$  First we consider for each $t\in [0, \infty)$ the stochastic process $\{\mathcal X(s,t):s\in\mathbb S\}$. 
\begin{prp}\label{prp:samplepaths1}
If  $d\in (1/2, 1)$ and the integrals
$$
\int_{\mathbb S}\frac{\sigma^2(v)}{[1-d(v)]^2}\mu(\mathrm dv)\quad\text{and}\quad\int_{\mathbb S}\frac{\sigma^2(v)}{[1-d(v)][2d(v)-1]}\mu(\mathrm dv)
$$
are finite, then for each $t\in [0, \infty)$ the stochastic process $\{\mathcal X(s, t):s\in \mathbb S\}$ has sample paths in $\mathcal L_2(\mu)$
and induces a Gaussian random element  with values in  $L_2(\mu)$ which is denoted  by $\mathcal X(\cdot,t)$. Moreover, the process $\{\mathcal X(\cdot,t): t\in [0, \infty)\}$ with values in $L_2(\mu)$ is Gaussian.
\end{prp}
\begin{proof}
Since we have that
\begin{align*}
\operatorname E\int_\mathbb S\mathcal X^2(v,t)\mu(\mathrm{d}v)
	&=\int_\mathbb S\frac{\sigma^2(v)c(v)}{[1-d(v)][3-2d(v)]}\cdot t^{3-2d(v)}\mu(\mathrm dv)\\
	&\le\max\{t,t^2\}\Bigl[\int_\mathbb S\frac{\sigma^2(v)}{[1-d(v)]^2}\mu(\mathrm dv)+\int_\mathbb S\frac{\sigma^2(v)}{[1-d(v)][2d(v)-1]}\mu(\mathrm dv)\Bigr]
\end{align*}
using inequality \eqref{eq:c(s)} to estimate $c(s)$ from above, the sample paths of the stochastic process $\{\mathcal X(s,t):s\in\mathbb S\}$ almost surely belong to the space $\mathcal L_2(\mu)$ for each $t\in[0,\infty)$. Hence $\mathcal X(\cdot,t)$ is a random element in $L_2(\mu)$. Clearly it is a Gaussian one.
\end{proof}

Finally, we show that the stochastic process $\{\mathcal X(\cdot,t):t\in[0,\infty)\}$ is operator self-similar.
\begin{prp}\label{prp:ss}
The stochastic process $\{\mathcal X(\cdot,t):t\in[0,\infty)\}$ is operator self-similar with scaling family of operators $\{a^H:a>0\}$ where $a^H$, $a>0$, is a multiplication operator defined by $a^Hf=\{a^{3/2-d(s)}f(s):s\in\mathbb S\}$ for  $f\in L_2(\mu)$.
\end{prp}
\begin{proof}
We need to show that
\begin{equation}\label{eq:randomprocesses}
\{\mathcal X(\cdot,at):t\in[0,\infty)\}\overset{fdd}=\{a^H\mathcal X(\cdot,t):t\in[0,\infty)\}.
\end{equation}

Since stochastic processes on both sides of equality \eqref{eq:randomprocesses} are zero-mean Gaussian stochastic processes, we only need to show that their covariance structure is the same. Using the fact that two operators $A$ and $B$ are equal if and only if $\langle Af,g\rangle=\langle Bf,g\rangle$ for all $f,g\in L_2(\mu)$ and the fact that
\begin{equation}
\label{eq:sscov}
\operatorname E[a^{3/2-d(r)}\mathcal X(r,t)a^{3/2-d(s)}\mathcal X(s,u)]=E[\mathcal X(r,at)\mathcal X(s,au)]
\end{equation}
for all $r,s\in\mathbb S$ and $t,u\in[0,\infty)$ (equality \eqref{eq:sscov} follows from equation~\eqref{eq:covV}),
we conclude the proof by showing that
\begin{align*}
\langle\operatorname E[\langle a^H\mathcal X(\cdot,t),f\rangle a^H\mathcal X(\cdot,u)],g\rangle
	&=\int_\mathbb S\operatorname E\biggl[\biggl(\int_\mathbb S a^{3/2-d(u)}\mathcal X(u,t)f(u)\mu(\mathrm du)\biggr)a^{3/2-d(r)}\mathcal X(r,u)\biggr]g(r)\mu(\mathrm dr)\\
	&=\int_\mathbb S\biggl(\int_\mathbb S\operatorname E[a^{3/2-d(u)}\mathcal X(u,t)a^{3/2-d(r)}\mathcal X(r,u)]f(u)\mu(\mathrm du)\biggr)g(r)\mu(\mathrm dr)\\
	&=\langle\operatorname E[\langle\mathcal X(\cdot,at),f\rangle\mathcal X(\cdot,au)],g\rangle
\end{align*}
for all $f,g\in L_2(\mu)$.
\end{proof}

\section{Main results}\label{sec:fclt}
\subsection{Functional central limit theorem}
We shall consider $\{\zeta_n\}$ as random elements in a separable Banach space $C([0,1]; L_2(\mu))$ of continuous functions $f:[0,1]\to L_2(\mu)$ endowed with the norm
$$
\|f\|=\sup_{t\in[0,1]}\Bigl[\int_\mathbb Sf^2(v,t)\mu(\mathrm dv)\Bigr]^{1/2},\quad f\in C([0,1]; L_2(\mu)).
$$

Before stating sufficient conditions for the functional central limit theorem, we define the limit Gaussian processes
$$
\mathcal G=\{\mathcal G(s,t):(s,t)\in\mathbb S\times[0,1]\}\quad\text{and}\quad \mathcal G'=\{\mathcal G'(s,t):(s,t)\in\mathbb S\times[0,1]\}.
$$
Let the stochastic process $\mathcal G$ be a restriction to $\mathbb S\times[0,1]$ of the stochastic process $\mathcal X=\{\mathcal{X}(s,t):(s,t)\in\mathbb S\times [0, \infty)\}$ defined in Section \ref{sec:oss}. Let the stochastic process $\mathcal G'$ be Gaussian with  the covariance function $\operatorname E[\mathcal G'(r,t)\mathcal G'(s,u)]=\sigma(r,s)\min(t,u)$, $(r,t),(s,u)\in\mathbb S\times[0,1]$. If the integral $\int_\mathbb S\sigma^2(v)\mu(\mathrm dv)$ is finite, then for each $t\in[0,1]$ the sample paths of the stochastic process $\{\mathcal G'(s,t):s\in\mathbb S\}$ belong to the space $\mathcal L_2(\mu)$ (the proof is basically the same as the proof of Proposition~\ref{prp:samplepaths1}).

The following proposition establishes conditions under which  both of the stochastic processes $\{\mathcal G(\cdot,t):t\in[0,1]\}$ and $\{\mathcal G'(\cdot,t):t\in [0,1]\}$ with values in the space $L_2(\mu)$  have continuous versions.
\begin{prp}
If the integrals
$$
\int_{\mathbb S}\frac{\sigma^2(v)}{[1-d(v)]^2}\mu(\mathrm dv)\quad\text{and}\quad\int_{\mathbb S}\frac{\sigma^2(v)}{[1-d(v)][2d(v)-1]}\mu(\mathrm dv)
$$
are finite, then the $L_2(\mu)$-valued stochastic process  $\{\mathcal G(\cdot, t):t\in [0, 1]\}$ has a continuous version.

If the integral
$$\int_\mathbb S\sigma^2(v)\mu(\mathrm dv)$$
is finite, then the $L_2(\mu)$-valued stochastic process $\{\mathcal G'(\cdot, t):t\in [0, 1]\}$  has a continuous version.
\end{prp}
\begin{proof}
We use the following inequality for the moments of a Gaussian random element $\xi$ with values in a separable Banach space:
\begin{equation}
\label{eq:gaussianmoments}
(\operatorname E\|\xi\|^p)^{1/p}\le K_{p,q}(\operatorname E\|\xi\|^q)^{1/q},
\end{equation}
where $0<p,q<\infty$ and $K_{p,q}$ is a constant depending on $p$ and $q$ only (for the proof, see Ledoux and Talagrand~\cite{ledoux1991}, p.~59, Corollary~3.2).

Using Kolmogorov's continuity theorem (see the textbook by Kallenberg~\cite{kallenberg1997},  p.~35, Theorem~2.23), inequality \eqref{eq:c(s)} to estimate $c(s)$ from above and inequalities
\begin{align*}
\operatorname E\|\mathcal G(\cdot,t)-\mathcal G(\cdot,u)\|^4
	&\le K_{4,2}^4(\operatorname E\|\mathcal G(\cdot,t)-\mathcal G(\cdot,u)\|^2)^2\\
	&<K_{4,2}^4\Bigl[\int_{\mathbb S}\frac{\sigma^2(v)}{[1-d(v)]^2}\mu(\mathrm dv)+\int_{\mathbb S}\frac{\sigma^2(v)}{[1-d(v)][2d(v)-1]}\mu(\mathrm dv)\Bigr]^2\cdot|t-u|^2
\end{align*}
and
$$
\operatorname E\|\mathcal G'(\cdot,t)-\mathcal G'(\cdot,u)\|^4\le K_{4,2}^4\Bigl[\int_\mathbb S\sigma^2(v)\mu(\mathrm dv)\Bigr]^2|t-u|^2,
$$
we conclude that the processes $\{\mathcal G(\cdot, t), t\in [0, 1]\}$ and$\{\mathcal G'(\cdot, t), t\in [0, 1]\}$ have continuous versions.
\end{proof}
Passing to continuous versions, we thus consider Gaussian stochastic processes $\mathcal G$ and $\mathcal G'$ as Gaussian random elements in the space $C([0,1];L_2(\mu))$. Clearly $\mathcal G'$ is an $L_2(\mu)$-valued Wiener process.

Now we are ready to state our main results. As usual $\xrightarrow{\mathcal D}$ denotes the convergence in distribution.
\begin{thm}\label{thm:fclt1}
Suppose that $d\in(1/2,1)$, the integrals
$$
\operatorname E\Bigl[\int_\mathbb S\frac{\varepsilon_0^2(v)}{[1-d(v)]^2}\mu(\mathrm dv)\Bigr]^{p/2}\quad\text{and}\quad\int_\mathbb S\frac{\sigma^2(v)}{[1-d(v)][2d(v)-1]}\mu(\mathrm dv)
$$
are finite and either $p=2$ and $\bar d=\esssup d<1$ or $p>2$. Then we have that
$$n^{-H}\zeta_n\xrightarrow{\mathcal D}\mathcal G\quad\text{as}\quad n\to\infty$$
in the space  $C([0,1];L_2(\mu))$, where $\{n^{-H}\}$ is a sequence of multiplication operators given by $n^{-H}f=\{n^{-[3/2-d(s)]}f(s):s\in\mathbb S\}$ for  $f\in L_2(\mu).$
\end{thm}
\begin{remark}
We have that, for $d\in(1/2,1)$ and $p>0$,
$$
\operatorname E\|\varepsilon_0\|^p
	<2^{-p}\operatorname E\Bigl[\int_\mathbb S\frac{\varepsilon_0^2(v)}{[1-d(v)]^2}\mu(\mathrm dv)\Bigr]^{p/2}
$$
since $1-d(v)<1/2$.
\end{remark}
\begin{thm}\label{thm:fclt2}
Suppose that $d=1$ and $\operatorname E\|\varepsilon_0\|^p<\infty$ for some $p>2$. Then we have that
$$(\sqrt n\log n)^{-1}\zeta_n\xrightarrow {\mathcal D}\mathcal G'\quad\text{as}\quad n\to\infty$$
in the space $C([0,1];L_2(\mu))$.
\end{thm}
\begin{thm}\label{thm:fclt3}
Suppose that $\underline d=\essinf d>1$ and $\operatorname E\|\varepsilon_0\|^2<\infty$. Then we have that
$$
(\sqrt n)^{-1}\zeta_n\xrightarrow {\mathcal D}\mathcal G'\quad\text{as}\quad n\to\infty
$$
in the space  $C([0,1];L_2(\mu))$.
\end{thm}

\begin{proof}[Proof of Theorem~\ref{thm:fclt3}]
The convergence of Theorem~\ref{thm:fclt3} follows from Theorem~5 of Ra\v ckauskas and Suquet~\cite{rackauskas2010} since $\sum_{j=0}^\infty\|u_j\|=\sum_{j=1}^\infty j^{-\underline d}<\infty$.
\end{proof}

\subsection{Proof of Theorem~\ref{thm:fclt1} and Theorem~\ref{thm:fclt2}}
The proof contains two major parts. We prove the convergence of the finite-dimensional distributions of the sequences $\{n^{-H}\zeta_n\}$ and $\{(\sqrt n\log n)^{-1}\zeta_n\}$ in the first part and we prove the tightness of these sequences in the second part.

To avoid considerations of two separate but similar cases, we denote $b_n^{-1}=n^{-H}$ and $\zeta=\mathcal G$ in the proof of Theorem~\ref{thm:fclt1}, whereas $b_n=\sqrt{n}\log n$ and $\zeta=\mathcal G'$ in the proof of Theorem~\ref{thm:fclt2}.

\subsubsection*{Convergence of the finite-dimensional distributions}
The convergence of the finite-dimensional distributions means that the convergence
\begin{equation}\label{eq:fdd}
\left(\begin{array}{ccc}b_n^{-1}\zeta_n(t_1)&\ldots&b_n^{-1}\zeta_n(t_q)\end{array}\right)\xrightarrow{ \mathcal D}\left(\begin{array}{ccc}\zeta(t_1)&\ldots&\zeta(t_q)\end{array}\right)
\end{equation}
holds in the space  $L_2^q(\mu)$ for all $q\in\mathbb N$ and for all $t_1,\ldots,t_q\in[0,1]$. Note that the space $L_2^q(\mu)$ is isomorphic to $L_2(\mu; \mathbb R^q)$, the space of $\mathbb R^q$-valued square $\mu$-integrable functions with the norm  
$$
\|f\|=\Bigl[\int_\mathbb S\|f(v)\|^2\mu(\mathrm dv)\Bigr]^{1/2},\quad f\in L_2(\mu;\mathbb R^q),
$$
where $||f(v)||$ denotes the Euclidean norm in $\mathbb R^q$.

Fix $t_1,\ldots,t_q\in[0,1]$ and denote, for $s\in \mathbb S$,
$$
\zeta_n^{(q)}(s)=(\zeta_n(t_1,s),\ldots,\zeta_n(t_q,s))^\mathrm T\quad\text{and}\quad\zeta^{(q)}(s)=(\zeta(t_1,s), \ldots, \zeta(t_q,s))^\mathrm T,
$$
where $\mathbf x^\mathrm T$ denotes transpose of a vector $\mathbf x$.

Let $\zeta_n^{(q)}=\{\zeta_n^{(q)}(s):s\in\mathbb S\}$ and $\zeta^{(q)}=\{\zeta^{(q)}(s):s\in\mathbb S\}$. We need to prove that 
\begin{equation}\label{eq:fdd2}
b_n^{-1}\zeta^{(q)}_n\xrightarrow {\mathcal{D}} \zeta^{(q)}
\end{equation}
in the space  $L_2(\mu; \mathbb R^q)$ to establish \eqref{eq:fdd}.

According to Theorem~2 in Cremers and Kadelka~\cite{cremers1986}, it suffices to prove the following:
\begin{enumerate}[(I)]
\item\label{item:fdd} there exists a measurable set $\mathbb S_0\subset\mathbb S$ such that $\mu(\mathbb S\setminus\mathbb S_0)=0$ and for any $p\in\mathbb N$ and $s_1,\ldots,s_p\in\mathbb S_0$ we have that
$$\bigl(\begin{array}{ccc}b_n^{-1}\zeta_n^{(q)}(s_1)&\ldots&b_n^{-1}\zeta_n^{(q)}(s_p)\end{array}\bigr)\xrightarrow{\mathcal D}\bigl(\begin{array}{ccc}\zeta^{(q)}(s_1)&\ldots&\zeta^{(q)}(s_p)\end{array}\bigr);$$
\item\label{item:tightness}
\begin{enumerate}[(a)]
\item\label{partA}for each $s\in\mathbb S$,
$$\operatorname E\bigl\|b_n^{-1}\zeta_n^{(q)}(s)\bigr\|^2\to\operatorname E\bigl\|\mathcal \zeta^{(q)}(s)\bigr\|^2;$$
\item\label{partB}there exists a $\mu$-integrable function $f:\mathbb S\to[0,\infty)$ such that for each $s\in\mathbb S$ and each $n\in\mathbb N$
$$\operatorname E\bigl\|b_n^{-1}\zeta_n^{(q)}(s)\bigr\|^2\le f(s).$$
\end{enumerate}
\end{enumerate}

We use an auxiliary result to prove \eqref{item:fdd} which is stated in Lemma \ref{lemma:rho} below and may be explained as follows.

Let $\mathbb E$ and $\mathbb F$ be two separable Hilbert spaces and let $L(\mathbb E,\mathbb F)$ be the space of bounded linear operators from $\mathbb E$ to $\mathbb F$. Suppose that a sequence $\{Z_n\}$ of $\mathbb F$-valued random elements can be expressed as
\begin{equation*}
Z_n=\sum_{j=-\infty}^\infty B_{nj}\xi_j,
\end{equation*}
where $\{B_{nj}\}$ is a sequence in $L(\mathbb E,\mathbb F)$ for each $n\in\mathbb N$ and $\{\xi_j\}$ is a sequence of independent and identically distributed $\mathbb E$-valued random elements with $\operatorname E\xi_0=0$ and $\operatorname E\|\xi_0\|^2<\infty$. Using the same linear bounded operators $\{B_{nj}\}$, we construct another sequence $\{\tilde Z_n\}$ of $\mathbb F$-valued random elements  that can be represented as
\begin{equation*}
\tilde Z_n=\sum_{j=-\infty}^\infty B_{nj}\tilde\xi_j,
\end{equation*}
where $\{\tilde\xi_j\}$ is a sequence of independent and identically distributed $\mathbb E$-valued Gaussian random elements with $\operatorname E\tilde\xi_0=0$ and the same covariance operator as that of $\xi_0$.

Under the conditions of Lemma \ref{lemma:rho} below, the sequences $\{Z_n\}$ and $\{\tilde Z_n\}$ have the same limiting behaviour, i.e.\ if one converges in distribution then so does the other and their limits coincide. Before we state Lemma \ref{lemma:rho}, we define the distance function $\rho_k$.
\begin{dfn*}
Let $U$ and $V$ be random elements with values in a separable Hilbert space $\mathbb H$. The distance function $\rho_k$ is given by
$$\rho_k(U,V)=\sup_{f\in F_k}|\operatorname Ef(U)-\operatorname Ef(V)|,$$
where $F_k$ is the set of $k$ times Frechet differentiable functions $f:\mathbb H\to\mathbb R$ such that
$$\sup_{x\in\mathbb H}\bigl|f^{(j)}(x)\bigr|\le1,\quad j=0,1,\ldots,k.$$
\end{dfn*}
It is proved in the paper by Gin\'e and Le\'on~\cite{gine1980} that, for every $k>0$,  the distance function $\rho_k$ metrizes the convergence in distribution of sequences of random elements with values in $\mathbb H$.

\begin{lemma}\label{lemma:rho}
If both of the conditions
\begin{equation}\label{eq:opnorm}
\lim_{n\to\infty}\sup_{j\in\mathbb Z}\|B_{nj}\|=0\quad\text{and}\quad\limsup_{n\to\infty}\sum_{j=-\infty}^\infty\|B_{nj}\|^2<\infty
\end{equation}
are satisfied, then $\lim_{n\to\infty}\rho_3(Z_n,\tilde Z_n)=0$.
\end{lemma}
\begin{proof}
The proof follows from the proof of Proposition 4.1 of Ra\v ckauskas and Suquet~\cite{rackauskas2011}. The only difference is that $\mathbb E=\mathbb F$ in Ra\v ckauskas and Suquet~\cite{rackauskas2011}, but the proof remains valid as long as
$$
\|B_{nk}\|\le\|B_{nk}\|\|f\|
$$
for each $n\in\mathbb N$, each $k\in\mathbb Z$ and each $f\in\mathbb E$.
\end{proof}

Let $s_1,\ldots,s_p\in\mathbb S$. We express the sequence 
$\{\left(\begin{array}{ccc}b_n^{-1}\zeta_n^q(s_1)&\ldots&b_n^{-1}\zeta_n^q(s_p)\end{array}\right)\}$
of random matrices as
\begin{align*}
(\begin{array}{ccc}b_n^{-1}\zeta_n^q(s_1)&\ldots&b_n^{-1}\zeta_n^q(s_p)\end{array})
	&=\sum_{j=-\infty}^\infty\left(\begin{array}{ccc}z_n^{-1}(s_1)a_{nj}(s_1,t_1)\varepsilon_j(s_1)&\cdots&z_n^{-1}(s_p)a_{nj}(s_p,t_1)\varepsilon_j(s_p)\\\vdots &\ddots&\vdots\\z_n^{-1}(s_1)a_{nj}(s_1,t_q)\varepsilon_j(s_1)&\cdots&z_n^{-1}(s_p)a_{nj}(s_p,t_q)\varepsilon_j(s_p)\end{array}\right)\\
	&=\sum_{j=-\infty}^\infty A_{nj}\mathcal E_j,
\end{align*}
where
$$
A_{nj}=\left(\begin{array}{ccc}z_n^{-1}(s_1)a_{nj}(s_1,t_1)&\cdots&z_n^{-1}(s_p)a_{nj}(s_p,t_1)\\\vdots&\ddots&\vdots\\z_n^{-1}(s_1)a_{nj}(s_1,t_q)&\cdots&z_n^{-1}(s_p)a_{nj}(s_p,t_q)\end{array}\right),\quad\mathcal E_j=\operatorname{diag}(\begin{array}{ccc}\varepsilon_j(s_1)&\ldots&\varepsilon_j(s_p)\end{array})
$$
and
$$
z_n(s)=
\begin{cases}
n^{3/2-d(s)},&\mbox{ if }1/2<d(s)<1;\\
\sqrt n\log n,&\mbox{ if }d(s)=1.
\end{cases}
$$

If $d\in(1/2,1]$, then the matrices $\{A_{nj}\}$ satisfy both of conditions \eqref{eq:opnorm}.
Indeed, since
$$
\sup_{j\in\mathbb Z}a_{nj}(s,t)=a_{n1}(s,t)=\sum_{k=1}^{\fl{nt}}k^{-d(s)}+\{nt\}(\fl{nt}+1)^{-d(s)},
$$
we have the following asymptotic relations
$$
\sup_{j\in\mathbb Z}a_{nj}(s,t)\sim
\begin{cases}
\frac{t^{1-d(s)}}{1-d(s)}\cdot n^{1-d(s)},&\mbox{ if }d(s)<1;\\
\log n,&\mbox{ if }d(s)=1.
\end{cases}
$$
We have that
$$
\operatorname E\zeta_n^2(s,t)=\sigma^2(s)\sum_{j=-\infty}^{\fl{nt}+1}a_{nj}^2(s,t)
$$
and we use the asymptotic behaviour of the variance $\operatorname E\zeta_n^2(s,t)$ (see Remark~\ref{remark:var}) to obtain the following asymptotic relations
$$
\sum_{j=-\infty}^{\fl{nt}+1}a_{nj}^2(s,t)\sim
\begin{cases}
\frac{c(s)}{[1-d(s)][3-2d(s)]}\cdot t^{3-2d(s)}\cdot n^{3-2d(s)},&\mbox{ if }1/2<d(s)<1;\\
t\cdot n\log^2n,&\mbox{ if }d(s)=1.
\end{cases}
$$

Now we investigate the sequence $\{\left(\begin{array}{ccc}b_n^{-1}\tilde\zeta_n^q(s_1)& \ldots&b_n^{-1}\tilde\zeta_n^q(s_p)\end{array}\right)\}$, which is expressed as
\begin{equation}\label{eq:gaussian}
\left(\begin{array}{ccc}b_n^{-1}\tilde\zeta_n^q(s_1)&\ldots&b_n^{-1}\tilde\zeta_n^q(s_p)\end{array}\right)=\sum_{j=-\infty}^\infty A_{nj}\tilde{\mathcal E}_j,
\end{equation}
where $\{\tilde{\mathcal E}_j\}$ is a sequence of independent and identically distributed Gaussian random matrices with zero mean and the same covariance operator as that of $\mathcal E_0$. Since $\{\left(\begin{array}{ccc}\tilde\zeta_n^q(r)&\ldots&\tilde\zeta_n^q(s_p)\end{array}\right)\}$ is a sequence of finite-dimensional Gaussian random elements, we only need to check for each $i=1,\ldots,p$ and each $j=1,\ldots,q$ the convergence 
$$
z_n^{-1}(s_i)\operatorname E\zeta_n(s_i,t_j)\to\operatorname E\zeta(s_i,t_j).
$$
But this easily follows from Proposition~\ref{prp:cov} and Proposition~\ref{prp:cov2}. The proof of \eqref{item:fdd} is complete.

Next we prove \eqref{item:tightness}. We prove \eqref{partA} using equalities
$$
\operatorname E\|b_n^{-1}\zeta_n^q(s)\|^2=\sum_{i=1}^q\operatorname E[z_n^{-1}(s)\zeta_n(s,t_i)]^2\quad\text{and}\quad\operatorname E\|\mathcal \zeta^q(s)\|^2=\sum_{i=1}^q\operatorname E\mathcal\zeta^2(s,t_i)
$$
and Remark \ref{remark:var}.

An auxiliary result is needed to prove part $\eqref{partB}$.
\begin{prp}\label{prp:dom}
If $1/2<d(s)<1$, then
\begin{equation}\label{eq:dom}
\operatorname E[n^{-[3/2-d(s)]}\zeta_n(s,t)]^2\le g(s)=2[g_1(s)+g_2(s)+g_3(s)],
\end{equation}
for each $n\in\mathbb N$, where
$$
g_1(s)=\sigma^2(s)\Bigl[1+\frac1{2d(s)-1}\Bigr],\quad g_2(s)=\frac{\sigma^2(s)}{[1-d(s)]^2},\quad g_3(s)=\frac{\sigma^2(s)}{[1-d(s)][2d(s)-1]}
$$
and $c(s)$ is given by \eqref{eq:const}.

If $d=1$, then
\begin{equation}\label{eq:dom2}
\operatorname E[(\sqrt n\log n)^{-1}\zeta_n(s,t_i)]^2\le M\cdot \sigma^2(s),
\end{equation}
where $M$ is a positive constant.
\end{prp}
\begin{proof}
Expanding $\operatorname E\zeta_n^2(s,t)$ gives 
\begin{equation}\label{eq:var}
\operatorname E\zeta_n^2(s,t)=\fl{nt}\gamma_0(s)+2\sum_{k=1}^{\fl{nt}}(\fl{nt}-k)\gamma_k(s)+2\{nt\}\sum_{k=1}^{\fl{nt}}\gamma_k(s)+\{nt\}^2\gamma_0(s).
\end{equation}
Using expression \eqref{eq:cross-co} for cross-covariance, bounding series with integrals from above and using inequality~\eqref{eq:c(s)} leads to the following inequalities that complete the proof of inequality~\eqref{eq:dom}:
$$
\gamma_0(s)\le\sigma^2(s)\Bigl[1+\frac1{2d(s)-1}\Bigr],
$$
$$
\sum_{k=1}^{\fl{nt}}(\fl{nt}-k)\gamma_k(s)
	\le\frac12\Bigl[\frac{\sigma^2(s)}{[1-d(s)]^2}+\frac{\sigma^2(s)}{[1-d(s)][2d(s)-1]}\Bigr]\lfloor nt\rfloor^{3-2d(s)}
$$
and
$$
\sum_{k=1}^{\fl{nt}}\gamma_k(s)\le\frac12\Bigl[\frac{\sigma^2(s)}{[1-d(s)]^2}+\frac{\sigma^2(s)}{[1-d(s)][2d(s)-1]}\Bigr]\lfloor nt\rfloor^{2[1-d(s)]}.
$$

We argue as follows to prove inequality \eqref{eq:dom2}. By setting $r=s$ in expression~$\eqref{eq:cross-co}$, we see that the only term in expression \eqref{eq:var} that depends on $s$ is $\sigma^2(s)$ since $d(s)=1$ for each $s\in\mathbb S$. It follows that the sequence
$$
\frac1{\sigma^2(s)}\cdot\operatorname E[(\sqrt n\log n)^{-1}\zeta_n(s,t)]^2
$$
does not depend on $s$ and it is a convergent sequence (see Remark \ref{remark:var}). So it is bounded by some positive constant, say $M$.
\end{proof}

Now we can obtain the required function $f$ using Proposition~\ref{prp:dom}, the fact that $\operatorname E\|\mathcal \zeta^q(s)\|^2=\sum_{i=1}^q\operatorname E\mathcal\zeta^2(s,t_i)$ and setting
$$
f(s)=
\begin{cases}
q\cdot g(s),&\mbox{ if }d(s)<1;\\
qM\cdot\sigma^2(s),&\mbox{ if }d(s)=1.
\end{cases}
$$
The proof of \eqref{item:tightness} is complete. This completes the proof of the convergence of the finite dimensional distributions of the sequence $\{b_n^{-1}\zeta_n\}$.

\subsubsection*{Tightness}
To establish tightness of the sequence $\{b_n^{-1}\zeta_n\}$, we use the following adaptation of  Theorem~12.3 from  Billingsley~\cite{billingsley1968} (see also Proposition~4.2 in Ra\v ckauskas and Suquet~\cite{rackauskas2011}). 
\begin{prp}\label{prp:tightnessRS}
Let $\mathbb H$ be a separable Hilbert space. The sequence $\{Z_n\}$ of random elements of the space $C([0,1];\mathbb H)$ is tight if
\begin{enumerate}[(i)]
\item\label{list:tight1}$\{Z_n(t)\}$ is tight on $\mathbb H$ for every $t\in[0,1]$;
\item\label{list:tight2}there exists $\gamma\ge0$, $a>1$ and a continuous increasing function $F:[0,1]\to\mathbb R$ such that
$$P(\|Z_n(t)-Z_n(u)\|>\lambda)\le\lambda^{-\gamma}|F(t)-F(u)|^a.$$
\end{enumerate}
\end{prp}
It follows from Characiejus and Ra\v ckauskas~\cite{characiejus2013} that the sequence $\{b_n^{-1}S_n\}$ converges in distribution in $L_2(\mu)$. Hence the sequence $\{b_n^{-1}\zeta_n(t)\}$ also converges in distribution in $L_2(\mu)$ and the sequence $\{b_n^{-1}\zeta_n(t)\}$ is tight on $L_2(\mu)$ for every $t\in[0,1]$ and condition~\eqref{list:tight1} of Proposition~\ref{prp:tightnessRS} holds.

Now we show that condition~\eqref{list:tight2} of Proposition~\ref{prp:tightnessRS} holds for the sequence $\{b_n^{-1}\zeta_n\}$. By $C$ we denote a generic positive constant, not necessarily the same at different occurrences. We also denote
$$
\Delta_n^p(t,u)=\operatorname E\|b_n^{-1}[\zeta_n(t)-\zeta_n(u)]\|^p,
$$
where $p\ge2$, $t,u\in[0,1]$ and $n\ge1$.
\begin{prp}\label{prp:tightness}
Suppose that $d\in(1/2,1)$ and the integrals
$$
\int_\mathbb S\frac{\sigma^2(r)}{2d(r)-1}\mu(\mathrm dr)\qquad\text{and}\qquad\operatorname E\Bigl[\int_\mathbb S\frac{\varepsilon_0^2(r)}{[1-d(r)]^2}\mu(\mathrm dr)\Bigr]^{p/2},\quad p\ge2,
$$
are finite. Let $\bar d=\esssup d$. Then
\begin{equation}\label{eq:increments_1}
\Delta_n^p(t,u)\le C\cdot|t-u|^{(3-2\bar d)p/2},\quad n\ge1.
\end{equation}

Suppose that $d=1$ and $\operatorname E\|\varepsilon_0\|^p<\infty$ for $p\ge2$. Then
\begin{equation}\label{eq:increments_2}
\Delta_n^p(t,u)\le C\cdot|t-u|^{p/2},\quad n\ge2.
\end{equation}
\end{prp}

We recall that $\lfloor\cdot\rfloor$ is the floor function defined by $\lfloor x\rfloor=\max\{m\in\mathbb Z\mid m\le x\}$ for $x\in\mathbb R$, $\lceil\cdot\rceil$ is the ceiling function defined by $\lceil x\rceil=\min\{m\in\mathbb Z\mid m\ge x\}$ for $x\in\mathbb R$ and $\{x\}=x-\fl{x}$ is a fractional part of $x\in\mathbb R$. Observe that $\{x\}=0$ if and only if $x\in\mathbb Z$ and
$$
\lceil x\rceil-\lfloor x\rfloor=
\begin{cases}
0,&\mbox{ if }x\in\mathbb Z;\\
1,&\mbox{ if }x\in\mathbb R\setminus\mathbb Z.
\end{cases}
$$

We need an auxiliary lemma to prove Proposition~\ref{prp:tightness}.
\begin{lemma}\label{lemma:doublesum}
Let $0\le u<t\le1$, $n\ge 1$ and $\{nt\}=\{nu\}=0$.

If $d\in(1/2,1)$, then
\begin{equation}
n^{-[3-2d(s)]}\sum_{j=-\infty}^{nt}\Bigl[\sum_{k=nu+1}^{nt}v_{k-j}(s)\Bigr]^2
	\le\Bigl[\frac2{[1-d(s)]^2}+\frac1{2d(s)-1}\Bigr]\cdot|t-u|^{3-2d(s)}
\end{equation}
for $n\ge1$, where $v_j(s)$ is given by~\eqref{eq:weightsv}.

If $d=1$, then
\begin{equation}
(\sqrt n\log n)^{-p}\sum_{j=-\infty}^{nt}\Bigl[\sum_{k=nu+1}^{nt}v_{k-j}\Bigr]^p
	\le C\cdot|t-u|^{p/2},
\end{equation}
for $n\ge 2$ and $p\ge 2$, where $v_j$ is given by~\eqref{eq:opweightsv}.
\end{lemma}
\begin{proof}
We investigate the series
\begin{equation}\label{eq:seriespth}
\sum_{j=-\infty}^{nt}\Bigl[\sum_{k=nu+1}^{nt}v_{k-j}(s)\Bigr]^p
\end{equation}
with $p=2$ in the case of $d\in(1/2,1)$ and $p\ge2$ in the case of $d=1$. Let us split series \eqref{eq:seriespth} into two terms
\begin{equation}\label{eq:series>0}
\sum_{j=-\infty}^{nt}\Bigl[\sum_{k=nu+1}^{nt}v_{k-j}(s)\Bigr]^p
	=\sum_{j=-nu+1}^\infty\Bigl[\sum_{k=nu+1}^{nt}(k+j)^{-d(s)}\Bigr]^p
	+\sum_{j=nu+1}^{nt}\Bigl[\sum_{k=1}^{nt-j+1}k^{-d(s)}\Bigr]^p
\end{equation}
and then split the first term on the right-hand side of \eqref{eq:series>0} again into two terms
\begin{multline}\label{eq:seriesinfty}
\sum_{j=-nu+1}^\infty\Bigl[\sum_{k=nu+1}^{nt}(k+j)^{-d(s)}\Bigr]^p
	=\sum_{j=-nu+1}^{n(t-2u)}\Bigl[\sum_{k=nu+1}^{nt}(k+j)^{-d(s)}\Bigr]^p\\
		+\sum_{j=n(t-2u)+1}^\infty\Bigl[\sum_{k=nu+1}^{nt}(k+j)^{-d(s)}\Bigr]^p.
\end{multline}

The first term on the right-hand side of \eqref{eq:seriesinfty} is estimated from above in the following way:
$$
n^{-[3-2d(s)]}\sum_{j=-nu+1}^{n(t-2u)}\Bigl[\sum_{k=nu+1}^{nt}(k+j)^{-d(s)}\Bigr]^2
\le\frac{n|t-u|}{n^{3-2d(s)}}\Bigl[\sum_{k=nu+1}^{nt}(k-nu)^{-d(s)}\Bigr]^2\le\frac{|t-u|^{3-2d(s)}}{[1-d(s)]^2}
$$
if $d\in(1/2,1)$;
\begin{align*}
(\sqrt n\log n)^{-p}\sum_{j=-nu+1}^{n(t-2u)}\Bigl[\sum_{k=nu+1}^{nt}(k+j)^{-1}\Bigr]^p
	&\le\Bigl[\frac1{\log n}\sum_{k=nu+1}^{nt}(k-nu)^{-1}\Bigr]^p\cdot|t-u|^{p/2}\\
	&\le\Bigl[\frac{1+\log(n|t-u|)}{\log n}\Bigr]^p\cdot|t-u|^{p/2}
\end{align*}
if $d=1$ since $1/n\le|t-u|$ (otherwise $t=u$ because we assume that $\{nt\}=\{nu\}=0$). The second term on the right-hand side of \eqref{eq:seriesinfty} is estimated from above using the inequality
$$
\sum_{j=n(t-2u)+1}^\infty\Bigl[\sum_{k=nu+1}^{nt}(k+j)^{-d(s)}\Bigr]^p
	\le(n|t-u|)^p\sum_{j=n(t-2u)+1}^\infty (nu+j)^{-pd(s)}
	\le\frac{(n|t-u|)^{p+1-pd(s)}}{pd(s)-1}
$$
and observing that
$$
n^{-[3-2d(s)]}(n|t-u|)^{p+1-pd(s)}=|t-u|^{3-2d(s)}
$$
if $d\in(1/2,1)$ and $p=2$ and
$$
(\sqrt n\log n)^{-p}(n|t-u|)^{p+1-pd(s)}\le\frac{|t-u|^{p/2}}{\log^p2}
$$
if $d=1$ and $p\ge2$ since $1/n\le|t-u|$.

The second term on the right-hand side of \eqref{eq:series>0} is estimated in the following way:
$$
n^{-[3-2d(s)]}\sum_{j=nu+1}^{nt}\Bigl[\sum_{k=1}^{nt-j+1}k^{-d(s)}\Bigr]^2
	\le\frac1{[1-d(s)]^2[3-2d(s)]}\cdot|t-u|^{3-2d(s)}
$$
if $d\in(1/2,1)$ and
\begin{align*}
(\sqrt n\log n)^{-p}\sum_{j=nu+1}^{nt}\Bigl[\sum_{k=1}^{nt-j+1}k^{-1}\Bigr]^p
	&\le\frac1{n^{p/2}\log^pn}\sum_{j=nu+1}^{nt}[1+\log(nt-j+1)]^p\\
	&\le2\Bigl[\frac1{\log2}+\frac{\log(n|t-u|)}{\log n}\Bigr]^p\cdot|t-u|^{p/2}
\end{align*}
if $d=1$. The proof of Lemma~\ref{lemma:doublesum} is complete.
\end{proof}

Now we are ready to prove Proposition~\ref{prp:tightness}.
\begin{proof}[Proof of Proposition~\ref{prp:tightness}]
Let $t,u\in[0,1]$. There is no loss of generality by assuming that $t>u$. Set $t'=\fl{nt}/n$ and $u'=\lceil nu\rceil/n$,  so that $t,t'\in[\lfloor nt\rfloor/n,\lceil nt\rceil/n]$, $u,u'\in[\lfloor nu\rfloor/n,\lceil nu\rceil/n]$, $\{nt'\}=\{nu'\}=0$ and $|t'-u'|\le|t-u|$. Since
$$
\Delta_n^p(t,u)\le C[\Delta_n^p(t,t')+\Delta_n^p(t',u')+\Delta_n^p(u',u)],
$$
we can establish inequalities~\eqref{eq:increments_1} and~\eqref{eq:increments_2} by investigating two cases: either $t,u\in[\kappa/n,(\kappa+1)/n]$ for some $\kappa\in\{0,\ldots,n-1\}$ or $\{nt\}=\{nu\}=0$.

First, suppose that $t,u\in[\kappa/n,(\kappa+1)/n]$ for some $\kappa\in\{0,\ldots,n-1\}$. Then $|t-u|\le1/n$ and
$$
\zeta_n(t)-\zeta_n(u)=n|t-u|X_{\kappa+1},
$$
so that
$$
\Delta_n^p(t,u)
	\le [n|t-u|]^p\|n^{-H}\|^p\operatorname E\|X_0\|^p\le\operatorname E\|X_0\|^p\cdot|t-u|^{(3-2\bar d)p/2}
$$
if $d\in(1/2,1)$ and
$$
\Delta_n^p(t,u)
	=[n|t-u|]^p(\sqrt n\log n)^{-p}\operatorname E\|X_0\|^p
	\le\frac{\operatorname E\|X_0\|^p}{\log^p 2}\cdot|t-u|^{p/2}
$$
if $d=1$ and $n\ge2$.

We have that
\begin{equation}\label{ineq:landt}
\operatorname E\|X_0\|^p
	\le2^{p-1}\Bigl(C\frac p{\log p}\Bigr)^p\Bigl[(\operatorname E\|X_0\|^2)^{p/2}+\sum_{j=0}^\infty\operatorname E\|u_j\varepsilon_{k-j}\|^p\Bigr]
\end{equation}
by using a slight modification of the inequality stated in Theorem~6.20 of Ledoux and Talagrand~\cite{ledoux1991}. Since
$$
\operatorname E\|X_0\|^2\le\operatorname E\|\varepsilon_0\|^2+\int_\mathbb S\frac{\sigma^2(r)}{2d(r)-1}\mu(\mathrm dr)
$$
and $\sum_{j=0}^\infty\operatorname E\|u_j\varepsilon_{k-j}\|^p\le \operatorname E\|\varepsilon_0\|^p\sum_{j=1}^\infty j^{-p/2}$, we have that $\operatorname E\|X_0\|^p<\infty$.

Secondly, suppose that $\{nt\}=\{nu\}=0$. Then $|t-u|\ge1/n$ ($\Delta_n^p(t,u)=0$ if $\{nt\}=\{nu\}=0$ and $|t-u|<1/n$). The increment $b_n^{-1}[\zeta_n(t)-\zeta_n(u)]$ may be expressed as a series of independent $L_2(\mu)$-valued random elements
$$
b_n^{-1}[\zeta_n(t)-\zeta_n(u)]=\sum_{j=-\infty}^{nt}b_{n}^{-1}\sum_{k=nu+1}^{nt}v_{k-j}\varepsilon_j
$$
where $v_j$ is given by \eqref{eq:weightsv}. Using the same inequality as in \eqref{ineq:landt}, we have that
$$
\Delta_n^p(t,u)
	\le 2^{p-1}\Bigl(C\frac{p}{\log p}\Bigr)^p\Bigl[\Delta_n^p2^{p/2}(t,u)
	+\sum_{j=-\infty}^{nt}\operatorname E\Bigl\|b_n^{-1}\sum_{k=nu+1}^{nt}v_{k-j}\varepsilon_j\Bigr\|^p\,\Bigr].
$$

If $d\in(1/2,1)$, then we have that
\begin{equation}\label{eq:2nd<1}
\Delta_n^2(t,u)
	=\int_\mathbb S\sigma^2(r)n^{-[3-2d(r)]}\sum_{j=-\infty}^{nt}\Bigl[\sum_{k=nu+1}^{nt}v_{k-j}(r)\Bigr]^2\mu(\mathrm dr)
\end{equation}
and
\begin{equation}\label{eq:pth<1}
\sum_{j=-\infty}^{nt}\operatorname E\Bigl\|b_n^{-1}\sum_{k=nu+1}^{nt}v_{k-j}\varepsilon_j\Bigr\|^p
	=\sum_{j=-\infty}^{nt}\operatorname E\Bigl[\int_\mathbb Sn^{-[3-2d(r)]}\bigl|\sum_{k=nu+1}^{nt}v_{k-j}(r)\bigr|^2\varepsilon_j^2(r)\mu(\mathrm dr)\Bigr]^{p/2}.
\end{equation}
If $d=1$, then we obtain
\begin{equation}\label{eq:2nd=1}
\Delta_n^2(t,u)
	=\operatorname E\|\varepsilon_0\|^2(\sqrt{n}\log n)^{-2}\sum_{j=-\infty}^{nt}\Bigl[\sum_{k=nu+1}^{nt}v_{k-j}\Bigr]^2
\end{equation}
and
\begin{equation}\label{eq:pth=1}
\sum_{j=-\infty}^{nt}\operatorname E\Bigl\|b_n^{-1}\sum_{k=nu+1}^{nt}v_{k-j}\varepsilon_j\Bigr\|^pt
	=\operatorname E\|\varepsilon_0\|^p(\sqrt n\log n)^{-p}\sum_{j=-\infty}^{nt}\Bigl[\sum_{k=nu+1}^{nt}v_{k-j}\Bigr]^p.
\end{equation}
We estimate \eqref{eq:2nd<1}, \eqref{eq:2nd=1} and~\eqref{eq:pth=1} using Lemma~\ref{lemma:doublesum} and we need to estimate series \eqref{eq:pth<1} for $p>2$ when $d\in(1/2,1)$. As in \eqref{eq:series>0} and \eqref{eq:seriesinfty}, we split series \eqref{eq:pth<1} into three parts and estimate them from above separately. The estimation is essentially similar to the estimation of series \eqref{eq:seriespth}. Let us recall that we assume that $1/n\le|t-u|$ if $\{nt\}=\{nu\}=0$. The following inequalities are obtained:
$$
\sum_{j=-nu+1}^{n(t-2u)}\operatorname E\Bigl[\int_\mathbb S\frac{|\sum_{k=nu+1}^{nt}(k+j)^{-d(r)}|^2\varepsilon_j^2(r)}{n^{3-2d(r)}}\mu(\mathrm dr)\Bigr]^{p/2}
	\le\operatorname E\Bigl[\int_\mathbb S\frac{\varepsilon_0^2(r)}{[1-d(r)]^2}\mu(\mathrm dr)\Bigr]^{p/2}|t-u|^{(3-2\bar d)p/2}
$$
since
$$
\frac{\sum_{k=nu+1}^{nt}(k-nu)^{-d(r)}}{n^{1-d(r)}}\le\frac{|t-u|^{1-d(r)}}{1-d(r)};
$$
$$
\sum_{j=n(t-2u)+1}^\infty\operatorname E\Bigl[\int_\mathbb S\frac{|\sum_{k=nu+1}^{nt}(k+j)^{-d(r)}|^2\varepsilon_j^2(r)}{n^{3-2d(r)}}\mu(\mathrm dr)\Bigr]^{p/2}
	\le\frac{\operatorname E\|\varepsilon_0\|^p}{p/2-1}|t-u|^{(3-2\bar d)p/2}
$$
since
$$
\Bigl(\frac n{nu+j}\Bigr)^{2d(r)}
	=\Bigl(\frac n{n|t-u|}\Bigr)^{2d(r)}\Bigl(\frac {n|t-u|}{nu+j}\Bigr)^{2d(r)}
	\le n|t-u|^{1-2\bar d}(nu+j)^{-1}
$$
for $j\ge n(t-2u)+1$;
$$
\sum_{j=nu+1}^{nt}\operatorname E\Bigl[\int_\mathbb S\frac{|\sum_{k=1}^{nt-j+1}k^{-d(r)}|^2\varepsilon_j^2(r)}{n^{3-2d(r)}}\mu(\mathrm dr)\Bigr]^{p/2}\\
	\le\frac{2^{1+p(1-\bar d)}}{1+p(1-\bar d)}\operatorname E\Bigl[\int_\mathbb S\frac{\varepsilon_0^2(r)}{[1-d(r)]^2}\mu(\mathrm dr)\Bigr]^{p/2}|t-u|^{(3-\bar d)p/2}
$$
since
$$
\frac{\sum_{k=1}^{nt-j+1}k^{-d(r)}}{n^{1-d(r)}}
	\le\frac1{1-d(r)}\Bigl[\frac{nt-j+1}n\Bigr]^{1-d(r)}
	\le\frac1{1-d(r)}\Bigl[\frac{nt-j+1}n\Bigr]^{1-\bar d}
$$
for $nu+1\le j\le nt$.

The proof of Proposition \ref{prp:tightness} is complete.
\end{proof}
We established the convergence of the finite-dimensional distributions and the tightness of the sequence $\{b_n^{-1}\zeta_n\}$. The proof of Theorem~\ref{thm:fclt1} and Theorem~\ref{thm:fclt2} is complete.

\section*{Acknowledgments}
The authors would like to thank the anonymous reviewer for the valuable comments and suggestions that improved the quality of the paper. The research was partially supported by the Research Council of Lithuania, grant No.~MIP-053/2012.

\bibliography{short.bib,bibliography.bib}

\begin{thebibliography}{10}

\bibitem{billingsley1968}
P.~Billinglsey.
\newblock {\em Convergence of Probability Measures}.
\newblock Wiley, 1968.

\bibitem{characiejus2013}
V.~Characiejus and A.~Ra\v{c}kauskas.
\newblock The central limit theorem for a sequence of random processes with
  space-varying long memory.
\newblock {\em Lith. Math. J.}, 53(2):149--160, April 2013.

\bibitem{cremers1984}
H.~Cremers and D.~Kadelka.
\newblock On weak convergence of stochastic processes with {L}usin path spaces.
\newblock {\em Manuscripta Math.}, 45(2):115--125, June 1984.

\bibitem{cremers1986}
H.~Cremers and D.~Kadelka.
\newblock On weak convergence of integral functionals of stochastic processes
  with applications to processes taking paths in \mbox{$L_p^E$}.
\newblock {\em Stoch. Process. Appl.}, 21(2):305--317, 1986.

\bibitem{didier2011}
G.~Didier and V.~Pipiras.
\newblock Integral representations and properties of operator fractional
  {B}rownian motions.
\newblock {\em Bernoulli}, 17(1):1--33, February 2011.

\bibitem{embrechts2002}
P.~Embrechts and M.~Maejima.
\newblock {\em Selfsimilar Processes}.
\newblock Princeton Series in Applied Mathematics. Princeton University Press,
  2002.

\bibitem{gine1980}
E.~Gin\'e and J.~R. Le\'on.
\newblock On the central limit theorem in {H}ilbert space.
\newblock {\em Stochastica}, IV(1):43--71, 1980.

\bibitem{giraitis2012}
L.~Giraitis, H.~L. Koul, and D.~Surgailis.
\newblock {\em Large Sample Inference for Long Memory Processes}.
\newblock Imperial College Press, 2012.

\bibitem{hudson1982}
W.~N. Hudson and J.~D. Mason.
\newblock Operator-self-similar processes in a finite-dimensional space.
\newblock {\em Trans. Am. Math. Soc.}, 273(1):281--297, September 1982.

\bibitem{kallenberg1997}
O.~Kallenberg.
\newblock {\em Foundations of Modern Probability}.
\newblock Probability and Its Applications. Springer, 1997.

\bibitem{laha1981}
R.~Laha and V.~Rohatgi.
\newblock Operator self similar stochastic processes in \mbox{$\mathbb R^d$}.
\newblock {\em Stoch. Process. Appl.}, 12(1):73--84, October 1981.

\bibitem{lamperti1962}
J.~Lamperti.
\newblock Semi-stable stochastic processes.
\newblock {\em Trans. Am. Math. Soc.}, 104(1):62--78, July 1962.

\bibitem{lavancier2009}
F.~Lavancier, A.~Philippe, and D.~Surgailis.
\newblock Covariance function of vector self-similar processes.
\newblock {\em Stat. Probab. Lett.}, 79(23):2415--2421, December 2009.

\bibitem{ledoux1991}
M.~Ledoux and M.~Talagrand.
\newblock {\em Probability in Banach Spaces}.
\newblock Springer, 1991.

\bibitem{maejima1994}
M.~Maejima and J.~D. Mason.
\newblock Operator-self-similar stable processes.
\newblock {\em Stoch. Process. Appl.}, 54(1):139--163, November 1994.

\bibitem{matache2006}
M.~T. Matache and V.~Matache.
\newblock Operator self-similar processes on {B}anach spaces.
\newblock {\em J. Appl. Math. Stoch. Anal.}, 2006:1--18, 2006.
\newblock Article ID 82838.

\bibitem{merlevede1997}
F.~Merlev\`ede, M.~Peligrad, and S.~Utev.
\newblock Sharp conditions for the {CLT} of linear processes in a {H}ilbert
  space.
\newblock {\em J. Theor. Probab.}, 10(3):681--693, July 1997.

\bibitem{rackauskas2010}
A.~Ra{\v c}kauskas and {\relax Ch}.~Suquet.
\newblock On limit theorems for {B}anach-space-valued linear processes.
\newblock {\em Lith. Math. J.}, 50(1):71--87, 2010.

\bibitem{rackauskas2011}
A.~Ra{\v c}kauskas and {\relax Ch}.~Suquet.
\newblock Operator fractional {B}rownian motion as limit of polygonal line
  processes in {H}ilbert space.
\newblock {\em Stoch. Dyn.}, 11(1):49--70, 2011.

\end{thebibliography}
\bibliographystyle{abbrv}

\end{document}